\documentclass[a4paper]{amsart}
\usepackage{graphicx,amsmath,amsfonts,latexsym,amssymb,amsthm,mathrsfs, color}
\usepackage[latin1]{inputenc}
\newtheorem{thm}{Theorem}[section]
\newtheorem{proposition}[thm]{Proposition}

\newtheorem{theorem}[thm]{Theorem}

\newtheorem{remark}[thm]{Remark}

\begin{document}
\title[]{A scattering approach to a surface with hyperbolic cusp}
\author{Nikolaos Roidos}
\address{Institut f\"ur Analysis, Leibniz Universit\"at Hannover, Welfengarten 1, 30167 Hannover, Germany}
\email{roidos@math.uni-hannover.de}

\begin{abstract}
Let $X$ be a two-dimensional smooth manifold with boundary $S^{1}$ and $Y=[1,\infty)\times S^{1}$. We consider a family of complete surfaces arising by endowing $X\cup_{S^{1}}Y$ with a parameter dependent Riemannian metric, such that the restriction of the metric to $Y$ converges to the hyperbolic metric as a limit with respect to the parameter. We describe the associated spectral and scattering theory of the Laplacian for such a surface. We further show that on $Y$ the zero $S^{1}$-Fourier coefficient of the generalized eigenfunction of this Laplacian, as a family with respect to the parameter, approximates in a certain sense, for large values of the spectral parameter, the zero $S^{1}$-Fourier coefficient of the generalized eigenfunction of the Laplacian for the case of a surface with hyperbolic cusp.
\end{abstract}

\date{\today}

\date{\today}

\maketitle

\section{Introduction}

Consider a two-dimensional smooth manifold consisting of a compact part $X$ which is glued together with a non-compact end of the form $Y=[1,\infty)\times S^{1}$. Endow $X\cup_{S^1}Y$ with a Riemannian metric such that its restriction to $Y$ admits the warped product structure $dx^{2}+f^{2}(x)dy^{2}$, $(x,y)\in[1,\infty)\times S^{1}$, where $f$ is an appropriate function. In \cite{HRS}, the case of $f(x)=x^{-a}$, $a\in(0,\infty)$, has been studied (generalized cusp) as an interpolation between the case of $f(x)=1$ (cylindrical end) and of $f(x)=e^{-x}$ (hyperbolic cusp). Among other results, a meromorphic continuation of the resolvent of the Laplacian to the whole complex plane was obtained, a complete description of the generalized eigenfunctions of the Laplacian was given and the properties of the scattering matrix were shown. The whole theory coincides with the flat-cylindrical case when the parameter $a$ tends to zero. However, the same does not happen in the hyperbolic case when $a$ tends to infinity. In the sequel, we follow a similar consideration as in \cite{HRS}, but by using a different parametric family of metrics in order to achieve a more realistic scattering approach to the case of a surface with hyperbolic cusp. 

The family of metrics that we consider here is obtained by choosing $f(x)=(1+\frac{x}{a})^{-a}$, $a\in(0,\infty)$. Although for this choice of $f$ a part of the theory can be recovered by \cite{HRS} by changing variables $t=1+\frac{x}{a}$ and $h=a^{-1}y$, as a first step we do construct the generalized eigenfunction in detail in order to emphasize the sets of poles and the concrete formulas that appear. Furthermore, we show meromorphic continuation of the resolvent by following a different method compared to \cite{HRS}, that is based on the ideas in \cite{St}.

In our consideration, as the parameter $a$ tends to infinity we approach a surface with hyperbolic cusp from scattering point of view in the following sense: the zero coefficient of the Fourier expansion over $S^{1}$ of the generalized eigenfunction of the Laplacian induced by the above family of metrics (excluding possibly the scattering matrix part) converges uniformly on compact sets of $Y$ to the zero coefficient of the Fourier expansion over $S^{1}$ of the generalized eigenfunction of the hyperbolic Laplacian for large values of the spectral parameter.

The main property of the above metric is that after Hodge decomposing the $L^2$-space on the cusp, the eigenvalue equation of the Laplacian can be solved explicitly on the subspace of the {\em harmonic components} (i.e. the subspace that corresponds to the zero coefficient of the Fourier expansion over $S^{1}$). Therefore, the continuous part of the spectral theorem for the Dirichlet Laplacian on the cusp can be explicitly stated and a meromorphic continuation of the resolvent to the logarithmic cover of the complex plane $\mathbb{C}$ can be achieved. Then, returning to the original surface with cusp, standard gluing techniques (see e.g. \cite{Mu}) provide us the meromorphic continuation of the full resolvent of the Laplacian. The last in combination with the explicit geometric structure that we have on the cusp, leads to the construction of the generalized eigenfunction of the Laplacian as well as its explicit asymptotic expansion on the cusp itself and the definition of the scattering matrix. In addition, we obtain the functional equation of the scattering matrix. Finally, by using the above described machinery in combination with the properties of Bessel functions we proceed to the main approximation result.

\section{Spectral theory on the cusp}

Consider a two-dimensional smooth manifold $M=X\cup_{S^{1}} Y$, where $X$ is a compact two-dimensional smooth manifold with boundary $S^{1}=\mathbb{R}/\mathbb{Z}$ and $Y=[1,\infty) \times S^{1}$. Assume that $M$ is endowed with a Riemannian metric $g$ such that when it is restricted to $Y$ admits the warped product form
\begin{gather}\label{metric}
g|_{Y}=dx^{2}+(1+\frac{x}{a})^{-2a}dy^{2},
\end{gather}
where $(x,y)\in [1,\infty) \times S^{1}$ and $a\in(0,\infty)$ is a fixed parameter. The dependence of $g|_{X}$ by $a$ can be arbitrary, e.g. we can assume that $g|_{X}$ is constant over $a$ outside a collar part of the boundary. Let $\mathbb{M}=(M,g)$, $\mathbb{X}=(X,g|_{X})$ and $\mathbb{Y}=(Y,g|_{Y})$.

The Laplace operator induced by $g|_{Y}$ on $Y$, acting on the space $C_{0}^{\infty}(Y\backslash\partial Y)$ of smooth compactly supported functions on $Y$ with support away of the boundary $\{1\}\times S^{1}$, is given by
\begin{gather}\label{Delta}
\Delta_{\mathbb{Y}}=\partial_{x}^{2}-\frac{1}{1+\frac{x}{a}}\partial_{x}+(1+\frac{x}{a})^{2a}\partial_{y}^{2}.
\end{gather}
Let $L^{2}(\mathbb{Y})$ be the space of all functions on $\mathbb{Y}$ that are square integrable with respect to the Riemannian measure 
$$
d\mu_{g|_{Y}}=(1+\frac{x}{a})^{-a}dxdy
$$ 
induced by $g|_{Y}$. By imposing Dirichlet boundary condition at $\{1\}\times S^{1}$, let $\underline{\Delta}_{\mathbb{Y}}$ be the corresponding self-adjoint extension of $\Delta_{\mathbb{Y}}$ in $L^{2}(\mathbb{Y})$ (which is also the Friedrichs extension).

A function $u$ on $Y$ that belongs to $L^2(S^{1})$ for each $x\in[1,\infty)$ admits a Fourier expansion over $S^{1}$, namely
\begin{gather}\label{exp}
u(x,y)=\sum_{n\in\mathbb{Z}}u_{n}(x)e^{2\pi n iy}.
\end{gather}
Hence, the eigenvalue equation of $\Delta_{\mathbb{Y}}$, i.e.
\begin{gather}\label{ee}
(\Delta_{\mathbb{Y}}+\lambda)u=0,
\end{gather}
by separation of variables is equivalent to 
\begin{gather}\label{ee2}
u''_{n}(x)-\frac{1}{1+\frac{x}{a}}u'_{n}(x)+(\lambda-4\pi^{2}n^{2}(1+\frac{x}{a})^{2a})u_{n}(x)=0, \quad n\in\mathbb{Z}.
\end{gather}
According to (\ref{exp}), or equivalently after Hodge decomposing $L^{2}(S^{1})$, the space $L^{2}(\mathbb{Y})$ can be decomposed as follows
\begin{gather}\label{hd}
 L^{2}(\mathbb{Y})= L_{\mathcal{H}}^{2}(\mathbb{Y})\oplus (L_{\mathcal{H}}^{2}(\mathbb{Y}))^{\perp},
\end{gather}
where $L_{\mathcal{H}}^{2}(\mathbb{Y})\in \mathrm{Ker}(\partial_{y}^{2})$ and $(L_{\mathcal{H}}^{2}(\mathbb{Y}))^{\perp}$ is the orthogonal complement of $L_{\mathcal{H}}^{2}(\mathbb{Y})$ with respect to the $S^{1}$-inner product. More precisely we have that
\begin{gather}
 L^{2}(\mathbb{Y})=\overline{\bigoplus_{n\in\mathbb{Z}}L^{2}([1,\infty),(1+\frac{x}{a})^{-a}dx)\otimes e^{2\pi n iy}},
\end{gather}
\begin{gather}
 L_{\mathcal{H}}^{2}(\mathbb{Y})=L^{2}([1,\infty),(1+\frac{x}{a})^{-a}dx)\otimes 1
\end{gather}
and
\begin{gather}
 (L_{\mathcal{H}}^{2}(\mathbb{Y}))^{\perp}=\overline{\bigoplus_{n\in\mathbb{Z}\backslash\{0\}}L^{2}([1,\infty),(1+\frac{x}{a})^{-a}dx)\otimes e^{2\pi n iy}}.
\end{gather}

\begin{theorem}\label{ds}
The Laplacian $\underline{\Delta}_{\mathbb{Y}}$ restricted to the space $(L_{\mathcal{H}}^{2}(\mathbb{Y}))^{\perp}$ has discrete spectrum.
\end{theorem}
\begin{proof}
By changing $u\rightarrow (1+\frac{x}{a})^{\frac{a}{2}}u$ in (\ref{ee2}), for each $n\in\mathbb{Z}\backslash\{0\}$ the eigenvalue equation of the restriction of $\underline{\Delta}_{\mathbb{Y}}$ to each of the subspaces 
\begin{gather*}
L^{2}([1,\infty),(1+\frac{x}{a})^{-a}dx)\otimes e^{2\pi n iy}
\end{gather*}
becomes $(A+V_{n}-\lambda)u=0$, where
\begin{gather*}
A=-\partial_{x}^{2} \quad \mbox{and} \quad V_{n}=(\frac{1}{4}+\frac{1}{2a})(1+\frac{x}{a})^{-2}+4\pi^{2}n^{2}(1+\frac{x}{a})^{2a} \quad \text{in} \quad L^{2}([1,\infty),dx).
\end{gather*}
Since $V\rightarrow\infty$ as $x\rightarrow\infty$, by standard theory (see e.g. \cite[Theorem XIII.16]{RS}) each $A+V_{n}$ has discrete spectrum. Moreover, by the min-max principle (see e.g. \cite[Theorem XIII.1]{RS}) the spectral bounds are bounded from below by $4\pi^{2} n^{2}$.
\end{proof}

In order to proceed to the study of the continuous spectrum of the Laplacian, we recall some properties of Bessel functions. Let $\nu\in\mathbb{C}$ and denote by $J_{\nu}$, $Y_{\nu}$ the Bessel functions of order $\nu$ of first and second kind respectively. The cylinder function $G_{\nu}(\lambda,x)$ of order $\nu$ is defined for $\lambda\in\mathbb{C}$ and $x\in[1,\infty)$ by 
\begin{gather*}
G_{\nu}(\lambda,x)=Y_{\nu}(\lambda)J_{\nu}(\lambda x)-J_{\nu}(\lambda)Y_{\nu}(\lambda x).
\end{gather*}
\begin{flushleft}
We have then the following two results.
\end{flushleft}
\begin{theorem}\label{WT1}(Weber, see e.g. \cite[14.52]{Wa})
If for some function $f$ of real variable the following integral $\int_{0}^{\infty}f(\lambda)\sqrt{\lambda}d\lambda$ exists and is absolutely convergent, then for any real $\nu$ we have
\begin{gather*}
\int_{1}^{\infty}\left(\int_{0}^{\infty}f(\lambda)G_{\nu}(\lambda,x)G_{\nu}(r,x)\lambda d\lambda\right)xdx=\frac{J_{\nu}^{2}(r)+Y_{\nu}^{2}(r)}{2}\left(f(r+0)+f(r-0)\right),
\end{gather*}
provided that the positive number $r$ lies inside an interval in which $f$ has finite total variation. 
\end{theorem}

\begin{theorem}\label{WT}(Weber's inversion formula, see \cite{Ti})
If for some function $f$ of real variable the integral $\int_{1}^{\infty}f(x)\sqrt{x}dx$ exists and is absolutely convergent, then for any real $\nu$ we have
\[
\int_{0}^{\infty}\left(\int_{1}^{\infty}\frac{f(x)G_{\nu}(\lambda,x)G_{\nu}(\lambda,r)}{J_{\nu}^{2}(\lambda)+Y_{\nu}^{2}(\lambda)}xdx\right)\lambda d\lambda=\frac{f(r+0)+f(r-0)}{2},
\]
provided that the positive number $r$ lies inside an interval in which $f$ has finite total variation. 
\end{theorem} 

From the above orthogonality relations we can define for any real $\nu$ the Weber transform of any $f\in C_{0}^{\infty}([0,\infty))$ by
\begin{gather*}
\mathbb{W}_{\nu}[f](x)=\int_{0}^{\infty}f(\lambda)G_{\nu}(\lambda,x)\lambda d\lambda.
\end{gather*}
$\mathbb{W}_{\nu}$ can be easily extended to a bijective isometry from $L^2([0,\infty), (J_{\nu}^{2}(\lambda)+Y_{\nu}^{2}(\lambda))\lambda d\lambda)$ to $L^2([1,\infty), x\, dx)$ with inverse given by 
\begin{gather*}
\mathbb{W}_{\nu}^{-1}[g](\lambda)=\int_{1}^{\infty}\frac{g(x)G_{\nu}(\lambda,x)}{J_{\nu}^{2}(\lambda)+Y_{\nu}^{2}(\lambda)}xdx,
\end{gather*}
for any $g\in C_{0}^{\infty}([1,\infty))$ (see the Appendix of \cite{HRS} for details). 
\par
By (\ref{ee2}), the eigenvalue equation of the restriction of $\underline{\Delta}_{\mathbb{Y}}$ to the subspace of the harmonic components $L_{\mathcal{H}}^{2}(\mathbb{Y})$ is given by the following Bessel equation 
\begin{gather*}
u''(x)-\frac{1}{1+\frac{x}{a}}u'(x)+\lambda u(x)=0.
\end{gather*}
Hence, by the Dirichlet condition at $x=1$ the generalized $\lambda$-eigenfunctions are given by 
\begin{eqnarray*}
(a+x)^{\frac{a+1}{2}}\lefteqn{G_{\frac{a+1}{2}}((a+1)\sqrt{\lambda},\frac{a+x}{a+1})}\\
&=&(a+x)^{\frac{a+1}{2}}\Big(Y_{\frac{a+1}{2}}((a+1)\sqrt{\lambda})J_{\frac{a+1}{2}}((a+x)\sqrt{\lambda})-J_{\frac{a+1}{2}}((a+1)\sqrt{\lambda})Y_{\frac{a+1}{2}}((a+x)\sqrt{\lambda})\Big).
\end{eqnarray*}
Therefore the continuous part of the spectral theorem can be expressed in terms of the Weber transform as follows. 

\begin{theorem}(Spectral theorem - continuous part)
The domain of the restriction of the Laplacian $\underline{\Delta}_{\mathbb{Y}}$ to the space $L_{\mathcal{H}}^{2}(\mathbb{Y})$ of harmonic components on the boundary from decomposition \eqref{hd}, is given by 
\begin{eqnarray*}
\lefteqn{\mathcal{D}(\underline{\Delta}_{\mathbb{Y}}|_{L_{\mathcal{H}}^{2}(\mathbb{Y})})}\\
&=&\Big\{u(x)\in L^{2}([1,\infty),xdx) \, | \, \lambda^{2}\mathbb{W}_{\frac{a+1}{2}}^{-1}[y^{-\frac{a+1}{2}}u((a+1)y-a)](\lambda) \in L^2([0,\infty), (J_{b}^{2}(\lambda)+Y_{b}^{2}(\lambda))\lambda d\lambda)\Big\}.
\end{eqnarray*}
For $u\in \mathcal{D}(\underline{\Delta}_{\mathbb{Y}}|_{L_{\mathcal{H}}^{2}(\mathbb{Y})})$ we have that
\[
(\underline{\Delta}_{\mathbb{Y}}u)(x)=(\frac{a+x}{a+1})^{\frac{a+1}{2}}\mathbb{W}_{\frac{a+1}{2}}[(\frac{\lambda}{a+1})^{2}\mathbb{W}_{\frac{a+1}{2}}^{-1}[y^{-\frac{a+1}{2}}u((a+1)y-a)](\lambda)](\frac{a+x}{a+1}).
\]
Furthermore, for the spectrum of the Laplacian $\underline{\Delta}_{\mathbb{Y}}$ on $\mathbb{Y}$ we have that $\sigma_{\mathrm{sing}}(\underline{\Delta}_{\mathbb{Y}})=\emptyset$ and $\sigma_{\mathrm{ac}}(\underline{\Delta}_{\mathbb{Y}})=\sigma_{\mathrm{cont}}(\underline{\Delta}_{\mathbb{Y}})=[0,\infty)$.
\end{theorem}

According to the previous theorem, the restriction of the resolvent of the Laplacian $\underline{\Delta}_{\mathbb{Y}}$ to the space $L_{\mathcal{H}}^{2}(\mathbb{Y})$ is given for any $\mu\in \mathbb{C}\backslash [0,\infty)$ by the bounded map $L_{\mathcal{H}}^{2}(\mathbb{Y})\ni u\mapsto (\underline{\Delta}_{\mathbb{Y}}-\mu)^{-1}u\in L_{\mathcal{H}}^{2}(\mathbb{Y})$, such that 
\begin{eqnarray*}
\lefteqn{(\underline{\Delta}_{\mathbb{Y}}-\mu)^{-1}u(x)}\\
&=&(\frac{a+x}{a+1})^{\frac{a+1}{2}}\mathbb{W}_{\frac{a+1}{2}}[\frac{1}{(\frac{\lambda}{a+1})^2-\mu}\mathbb{W}_{\frac{a+1}{2}}^{-1}[y^{-\frac{a+1}{2}}u((a+1)y-a)](\lambda)](\frac{a+x}{a+1}).
\end{eqnarray*}
If we further restrict $u\in L_{\mathcal{H}}^{2}(\mathbb{Y}) \cap C_{0}^{\infty}(Y\backslash\partial Y)$, we get an integral representation of the resolvent, namely 
\begin{gather}\label{res}
(\underline{\Delta}_{\mathbb{Y}}-\mu)^{-1}u(x)=\int_{1}^{\infty}k(\mu,x,y)u(y)dx,
\end{gather}
with kernel
\begin{gather}\label{ker}
k(\mu,x,y)=(\frac{a+x}{a+y})^{\frac{a+1}{2}}(a+y)\int_{0}^{\infty}\frac{G_{\frac{a+1}{2}}(\lambda,\frac{a+x}{a+1})G_{\frac{a+1}{2}}(\lambda,\frac{a+y}{a+1})}{J_{\frac{a+1}{2}}^{2}(\lambda)+Y_{\frac{a+1}{2}}^{2}(\lambda)}\frac{\lambda}{\lambda^{2}-\mu(a+1)^{2}}d\lambda.
\end{gather}

From the asymptotic behavior of the Bessel functions (see e.g. \cite[Chapter VII]{Wa}), for any $r>0$ there exists some $c(r)>0$ such that
\begin{gather*}
\big|\frac{G_{\frac{a+1}{2}}(\lambda,\frac{a+x}{a+1})G_{\frac{a+1}{2}}(\lambda,\frac{a+y}{a+1})}{J_{\frac{a+1}{2}}^{2}(\lambda)+Y_{\frac{a+1}{2}}^{2}(\lambda)}\lambda\big|\leq c(r)\frac{a+1}{\sqrt{(a+x)(a+y)}}
\end{gather*}
with $\lambda\in[r,\infty)$ and $x,y\in[1,\infty)$. Furthermore, there exists some $c'(a)>0$ such that 
\begin{eqnarray*}
\lefteqn{\big|\frac{G_{\frac{a+1}{2}}(\lambda,\frac{a+x}{a+1})G_{\frac{a+1}{2}}(\lambda,\frac{a+y}{a+1})}{J_{\frac{a+1}{2}}^{2}(\lambda)+Y_{\frac{a+1}{2}}^{2}(\lambda)}\lambda\big|}\\
&\leq& c'(a)\big|(\frac{(a+x)(a+y)}{(a+1)^{2}})^{\frac{a+1}{2}}-(\frac{a+x}{a+y})^{\frac{a+1}{2}}-(\frac{a+y}{a+x})^{\frac{a+1}{2}}+(\frac{(a+x)(a+y)}{(a+1)^{2}})^{-\frac{a+1}{2}}\big|
\end{eqnarray*}
with $\lambda\in(0,r]$ and $x,y\in[1,\infty)$. Hence, the kernel \eqref{ker} is well defined. 

Moreover, by following the ideas in \cite{St}, $k(\cdot,x,y)$ can be meromorphically continued to the logarithmic cover of $\mathbb{C}$. More precisely, by letting $\mu=e^{z}$, $\mathrm{Im}(z)\in(0,2\pi)$, and changing variables in \eqref{ker}, we obtain that
\begin{gather}\label{ker2}
k(e^{z},x,y)=(\frac{a+x}{a+y})^{\frac{a+1}{2}}(a+y)\int_{-\infty}^{\infty}\frac{G_{\frac{a+1}{2}}((a+1)e^{w},\frac{a+x}{a+1})G_{\frac{a+1}{2}}((a+1)e^{w},\frac{a+y}{a+1})}{J_{\frac{a+1}{2}}^{2}((a+1)e^{w})+Y_{\frac{a+1}{2}}^{2}((a+1)e^{w})}\frac{e^{2w}}{e^{2w}-e^{z}}dw.
\end{gather}
Denote by $H_{\nu}^{(1)}(\lambda)=J_{\nu}(\lambda)+iY_{\nu}(\lambda)$ and $H_{\nu}^{(2)}(\lambda)=J_{\nu}(\lambda)-iY_{\nu}(\lambda)$, $\nu,\lambda\in\mathbb{C}$, the Hankel function of the first and second kind respectively of order $\nu$. Recall that the Hankel functions have simple zeros on the logarithmic cover of $\mathbb{C}$ and that the following identity holds 
$$
H_{\nu}^{(1)}(e^{z})H_{\nu}^{(2)}(e^{z})=J_{\nu}^{2}(e^{z})+Y_{\nu}^{2}(e^{z}), \quad \nu,z\in \mathbb{C}.
$$
For $a\in\mathbb{R}$ let the following discrete sets of points in $\mathbb{C}$, namely
$$
\mathcal{B}_{a}=\Big\{z\in\mathbb{C}\, |\, H_{\frac{a+1}{2}}^{(1)}((a+1)e^{z})H_{\frac{a+1}{2}}^{(2)}((a+1)e^{z})=0\Big\},
$$ 
$$
\mathcal{B}'_{a}=\Big\{z+2\pi i\in\mathbb{C}\, |\, H_{\frac{a+1}{2}}^{(1)}((a+1)e^{z})H_{\frac{a+1}{2}}^{(2)}((a+1)e^{z})=0\Big\}.
$$ 
and
\begin{gather}\label{B}
\mathcal{H}_{a}=\bigcup_{k\in\mathbb{Z}}\Big\{z\in\mathbb{C}\, |\, H_{\frac{a+1}{2}}^{(1)}((a+1)e^{\frac{z}{2}+k\pi i})H_{\frac{a+1}{2}}^{(2)}((a+1)e^{\frac{z}{2}+k\pi i})=0\quad\, \mbox{and}\quad\, \mathrm{Im}(z)\in\mathbb{R}\backslash(0,2\pi) \Big\}.
\end {gather}

If $z\in \mathbb{C}\backslash\mathcal{H}_{a}$ with $\mathrm{Im}(z)\leq 0$, then we can deform the path of integration in \eqref{ker2} from $\mathbb{R}$ to $\Gamma=(-\infty,\alpha]\cup \Lambda \cup [\beta,\infty)$, for some $\alpha,\beta\in\mathbb{R}$ with $\alpha<\beta$, where $\Lambda$ is any smooth simple curve running from $\alpha$ to $\beta$ such that $z\in\mathbb{C}\backslash\Gamma$ lies on the left of $\Gamma$. Then, by Cauchy's theorem, \eqref{ker2} implies 
\begin{eqnarray}\nonumber
\lefteqn{k(e^{z},x,y)}\\\nonumber
&=&(\frac{a+x}{a+y})^{\frac{a+1}{2}}(a+y)\Big(\int_{\Gamma}\frac{G_{\frac{a+1}{2}}((a+1)e^{w},\frac{a+x}{a+1})G_{\frac{a+1}{2}}((a+1)e^{w},\frac{a+y}{a+1})}{J_{\frac{a+1}{2}}^{2}((a+1)e^{w})+Y_{\frac{a+1}{2}}^{2}((a+1)e^{w})}\frac{e^{2w}}{e^{2w}-e^{z}}dw\\\nonumber
&&-2\pi i\sum_{w_{i}\in\Omega\cap \mathcal{B}_{a}}\big(\frac{e^{2w_{i}}G_{\frac{a+1}{2}}((a+1)e^{w_{i}},\frac{a+x}{a+1})G_{\frac{a+1}{2}}((a+1)e^{w_{i}},\frac{a+y}{a+1})}{e^{2w_{i}}-e^{z}}\big)\\\label{ker3}
&&\times\big(\lim_{w\rightarrow w_{i}}\frac{w-w_{i}}{H_{\frac{a+1}{2}}^{(1)}((a+1)e^{w})H_{\frac{a+1}{2}}^{(2)}((a+1)e^{w})}\big)\Big),
\end{eqnarray}
where by $\Omega$ we denote the area between $\Lambda$ and the real axis.

Similarly, by \eqref{ker2} we have that
\begin{gather}\label{ker4}
k(e^{z},x,y)=(\frac{a+x}{a+y})^{\frac{a+1}{2}}(a+y)\int_{\mathbb{R}+2\pi i}\frac{G_{\frac{a+1}{2}}((a+1)e^{w-2\pi i},\frac{a+x}{a+1})G_{\frac{a+1}{2}}((a+1)e^{w-2\pi i},\frac{a+y}{a+1})}{J_{\frac{a+1}{2}}^{2}((a+1)e^{w-2\pi i})+Y_{\frac{a+1}{2}}^{2}((a+1)e^{w-2\pi i})}\frac{e^{2w}}{e^{2w}-e^{z}}dw.
\end{gather}
Then, if $z\in \mathbb{C}\backslash\mathcal{H}_{a}$ with $\mathrm{Im}(z)\geq 2\pi$, we can deform the new path from $\mathbb{R}+2\pi i$ to $\Gamma'=(-\infty+2\pi i,\alpha'+2\pi i]\cup \Lambda' \cup [\beta'+2\pi i,\infty+2\pi i)$, for some $\alpha',\beta'\in\mathbb{R}$ with $\alpha'<\beta'$, and $\Lambda'$ to be any smooth simple curve running from $\alpha'+2\pi i$ to $\beta'+2\pi i$ such that $z\in\mathbb{C}\backslash\Gamma'$ lies on the right of $\Gamma'$. In this case, from \eqref{ker4} Cauchy's theorem implies 
\begin{eqnarray}\nonumber
\lefteqn{k(e^{z},x,y)}\\\nonumber
&=&(\frac{a+x}{a+y})^{\frac{a+1}{2}}(a+y)\Big(\int_{\Gamma'}\frac{G_{\frac{a+1}{2}}((a+1)e^{w-2\pi i},\frac{a+x}{a+1})G_{\frac{a+1}{2}}((a+1)e^{w-2\pi i},\frac{a+y}{a+1})}{J_{\frac{a+1}{2}}^{2}((a+1)e^{w-2\pi i})+Y_{\frac{a+1}{2}}^{2}((a+1)e^{w-2\pi i})}\frac{e^{2w}}{e^{2w}-e^{z}}dw\\\nonumber
&&+2\pi i\sum_{w_{i}\in\Omega'\cap \mathcal{B}'_{a}}\big(\frac{e^{2w_{i}}G_{\frac{a+1}{2}}((a+1)e^{w_{i}-2\pi i},\frac{a+x}{a+1})G_{\frac{a+1}{2}}((a+1)e^{w_{i}-2\pi i},\frac{a+y}{a+1})}{e^{2w_{i}}-e^{z}}\big)\\\label{ker5}
&&\times\big(\lim_{w\rightarrow w_{i}}\frac{w-w_{i}}{H_{\frac{a+1}{2}}^{(1)}((a+1)e^{w-2\pi i})H_{\frac{a+1}{2}}^{(2)}((a+1)e^{w-2\pi i})}\big)\Big),
\end{eqnarray}
where by $\Omega'$ we now denote the area between $\Lambda'$ and the axis $\{z\in \mathbb{C}\, |\, \mathrm{Im}(z)=2\pi \}$.

By \eqref{res}, \eqref{ker3} and \eqref{ker5} we see that the resolvent family $(\underline{\Delta}_{\mathbb{Y}}-e^{z})^{-1}$ of $\underline{\Delta}_{\mathbb{Y}}$ restricted to the space $L_{\mathcal{H}}^{2}(\mathbb{Y})$ admits a meromorphic continuation over $\mathbb{C}$ with simple poles that coincide with the set $\mathcal{H}_{a}$. Further, similarly to \cite[Section 2.2]{HRS}, by the asymptotic behavior of the Bessel functions we can easily see that this continuation is a bounded operator family from $e^{-x^2}\otimes L_{\mathcal{H}}^{2}(\mathbb{Y})$ to $e^{x^2}\otimes L_{\mathcal{H}}^{2}(\mathbb{Y})$ (i.e. not on the initial space $L_{\mathcal{H}}^{2}(\mathbb{Y})$). Finally, by Theorem \ref{ds}, the resolvent $(\underline{\Delta}_{\mathbb{Y}}-e^{z})^{-1}$ restricted to the space $(L_{\mathcal{H}}^{2}(\mathbb{Y}))^{\perp}$ can be meromorphically continued to $\mathbb{C}$ with poles lying on the set $\cup_{k\in\mathbb{Z}}\{z\in\mathbb{C}\, |\,\mathrm{Im}(z)=2k\pi\}$. We can therefore conclude the following. 
\begin{theorem}\label{mc}
The resolvent family $(\underline{\Delta}_{\mathbb{Y}}-e^{z})^{-1}\in\mathcal{L}(L^{2}(\mathbb{Y}))$, $\mathrm{Im}(z)\in(0,2\pi)$, admits a meromorphic continuation over $z$ to the whole complex plane $\mathbb{C}$ as an operator family in $\mathcal{L}(e^{-x^2}\otimes L^{2}(\mathbb{Y}),e^{x^2}\otimes L^{2}(\mathbb{Y}))$ with simple poles that consist of two sets: 
\begin{itemize}
\item[(i)] A discrete set of points that is contained in $\cup_{k\in\mathbb{Z}}\{z\in\mathbb{C}\, |\, \mathrm{Im}(z)=2k\pi\}$.
\item[(ii)] The set $\mathcal{H}_{a}$ defined in \eqref{B}.
\end{itemize}
\end{theorem}

\section{Spectral theory on the surface with cusp}

Let $L^{2}(\mathbb{M})$ be the space of all functions on $\mathbb{M}$ that are square integrable with respect to the Riemannian measure induced by the metric $g$. Denote by $\Delta_{\mathbb{M}}$ the Laplacian on $\mathbb{M}$ induced by $g$ and by $\underline{\Delta}_{\mathbb{M}}$ the unique closed self-adjoint extensions in $L^{2}(\mathbb{M})$ of $\Delta_{\mathbb{M}}$ in $C_{0}^{\infty}(M)$. Further, let $\mathbb{B}=(B,g_{B})$ be a closed (i.e. compact without boundary) Riemannian surface such that $\mathbb{X}\cup\{[1,2]\times S^{1}\}$ is isometrically embedded into $\mathbb{B}$. Let $\Delta_{\mathbb{B}}$ be the Laplacian on $\mathbb{B}$, and similarly denote by $\underline{\Delta}_{\mathbb{B}}$ the unique closed extensions in $L^{2}(\mathbb{B})$ of $\Delta_{\mathbb{B}}$ in $C^{\infty}(B)$. Since $\mathbb{B}$ is closed, the resolvent $(\underline{\Delta}_{\mathbb{B}}-e^{z})^{-1}$ is a meromorphic family over $z\in\mathbb{C}$ with simple poles in the set $\cup_{k\in\mathbb{Z}}\{z\in\mathbb{C}\, |\, \mathrm{Im}(z)=2k\pi\}$. 

Let $\chi_{2}\in C^{\infty}(M)$ such that $\chi_{2}(x)=0$ on $X\cup[1,\frac{4}{3}]\times S^{1}$ and $\chi_{2}(x)=1$ when $x\geq\frac{5}{3}$, and define $\chi_{1}=1-\chi_{2}$. Further, take $\psi_{1}\in C_{0}^{\infty}(M)$ such that $\psi_{1}=1$ on $X\cup[1,\frac{5}{3}]\times S^{1}$ and $\psi_{1}=0$ when $x\geq2$. Also, let $\psi_{2}\in C^{\infty}(M)$ such that $\psi_{2}=1$ when $x\geq\frac{4}{3}$ and $\psi_{2}=0$ on $X$. Finally, assume for simplicity that all functions $\chi_{1}$, $\chi_{2}$, $\psi_{1}$ and $\psi_{2}$ when restricted to $Y$, they depend only on the $x$ variable.

If we denote by $\tilde{e}^{\pm x^{2}}$ a smooth extension to $M$ of the function $e^{\pm x^{2}}$ on $Y$, then similarly to \cite[Theorem 1]{Mu} or \cite[Theorem 1.1]{HRS}) we have the following continuation result. 
\begin{theorem}\label{continuation}
The resolvent $(\underline{\Delta}_{\mathbb{M}}-e^{z})^{-1}\in\mathcal{L}(L^{2}(\mathbb{M}))$, $\mathrm{Im}(z)\in(0,2\pi)$, admits a meromorphic continuation over $z$ to the whole complex plane $\mathbb{C}$ as an operator family in $\mathcal{L}(\tilde{e}^{-x^2}\otimes L^{2}(\mathbb{M}),\tilde{e}^{x^2}\otimes L^{2}(\mathbb{M}))$ with simple poles that are contained in the following three sets: 
\begin{itemize}
\item[(i)] A discrete set of points that is contained in $\cup_{k\in\mathbb{Z}}\{z\in\mathbb{C}\, |\, \mathrm{Im}(z)=2k\pi\}$,
\item[(ii)] The set $\mathcal{H}_{a}$ defined in \eqref{B},
\item[(iii)] The poles of the meromorphic family over $z\in\mathbb{C}$, $(I+K_{a}(z))^{-1}$,
\end{itemize}
where 
\begin{gather}\label{kdef}
K_{a}(z)=[\underline{\Delta}_{\mathbb{M}},\psi_{1}](\underline{\Delta}_{\mathbb{B}}-e^{z})^{-1}\chi_{1}+[\underline{\Delta}_{\mathbb{M}},\psi_{2}](\underline{\Delta}_{\mathbb{Y}}-e^{z})^{-1}\chi_{2}
\end{gather}
is a meromorphic family over $z\in\mathbb{C}$ of compact operators with poles of finite rank.
\end{theorem}
\begin{proof}

We start by defining a parametrix of $(\underline{\Delta}_{\mathbb{M}}-e^{z})^{-1}$ by
\begin{gather*}
Q_{z}=\psi_{1}(\underline{\Delta}_{\mathbb{B}}-e^{z})^{-1}\chi_{1}+\psi_{2}(\underline{\Delta}_{\mathbb{Y}}-e^{z})^{-1}\chi_{2},
\end{gather*}
which is a meromorphic family over $z\in\mathbb{C}$ with values in $\mathcal{L}(\tilde{e}^{-x^2}\otimes L^{2}(\mathbb{M}),\tilde{e}^{x^2}\otimes L^{2}(\mathbb{M}))$. From the choice of the cut-off functions we have that 
\begin{gather*}\label{parametrix}
(\underline{\Delta}_{\mathbb{M}}-e^{z})Q_{z}=I+K_{a}(z).
\end{gather*}

The support of the kernel of $K_{a}(z)$ is disjoint from the diagonal and is compact in the left variable. Moreover, $K_{a}(z)$ has smooth kernel as a pseudodifferential operator. Thus, $K_{a}(z)$ is a meromorphic family of compact operators, and it is easy to see that its poles are of finite rank. The operators $\chi_{1}$, $\chi_{2}$, $[\underline{\Delta}_{\mathbb{M}},\psi_{1}]$ and $[\underline{\Delta}_{\mathbb{M}},\psi_{2}]$ are of order $\leq1$. Hence, by the standard decay properties of the resolvent of a sectorial operator in the interpolation space (see e.g. \cite[Corollary 2.4]{RS2}), we have that the norm of $K_{a}(z)$ tends to zero as $\mathrm{Re}(z)\rightarrow+\infty$ with $\mathrm{Im}(z)=\pi$. Thus, by the meromorphic Fredholm theorem (see e.g. \cite[XIII.13]{RS}), $(I+K_{a}(z))^{-1}$ is a meromorphic family over $z\in\mathbb{C}$ with values in $\mathcal{L}(\tilde{e}^{-x^2}\otimes L^{2}(\mathbb{M}))$ and with poles of finite rank. 

\end{proof}

Following the ideas in \cite{Mu}, let $\chi\in C^{\infty}(\mathbb{R})$ such that $\chi(x)=0$ if $x<1$ and $\chi(x)=1$ if $x>1+\varepsilon$, for some $\varepsilon>0$. When $z$ lies in the resolvent set of $\underline{\Delta}_{\mathbb{M}}$, i.e. in $0<\mathrm{Im}(z)<2\pi$, the Hankel function of the second kind $H_{\frac{a+1}{2}}^{(2)}((a+x)e^{\frac{z}{2}})$ does not belong to $L^{2}(\mathbb{Y})$. Thus, for any $z\in\mathbb{C}$ we can set 
\begin{gather}\label{geg}
E_{a,z}=\phi_{\chi}-(\underline{\Delta}_{\mathbb{M}}-e^{z})^{-1}(\Delta_{\mathbb{M}}-e^{z})\phi_{\chi},
\end{gather}
where 
\begin{gather*}
\phi_{\chi}=\Bigg \{\begin{array}{ll} 0 & \mbox{in} \quad X \\ \chi(x)(a+x)^{\frac{a+1}{2}}H_{\frac{a+1}{2}}^{(2)}((a+x)e^{\frac{z}{2}}) & \mbox{in} \quad Y \end{array}.
\end{gather*}

Since $(\Delta_{\mathbb{M}}-e^{z})\phi_{\chi}$ is identically zero when $x>1+\varepsilon$, we have that $(\Delta_{\mathbb{M}}-e^{z})\phi_{\chi}$ is smooth and compactly supported on $M$. Thus, $(\underline{\Delta}_{\mathbb{M}}-e^{z})^{-1}(\Delta_{\mathbb{M}}-e^{z})\phi_{\chi}$ is well defined and is not equal to $\phi_{\chi}$ as $\phi_{\chi}\not\in L^{2}(\mathbb{M})$. Hence, $E_{a,z}$ is well defined, is not equal to zero and is smooth on $M$ and meromorphic over $z\in\mathbb{C}$ with poles that coincide with the poles of $(\underline{\Delta}_{\mathbb{M}}-e^{z})^{-1}$. By restricting $\mathrm{Im}(z)\in(0,2\pi)$ we have that $(\Delta_{\mathbb{M}}-e^{z})E_{a,z}=0$, and hence by meromorphicity $(\Delta_{\mathbb{M}}-e^{z})E_{a,z}=0$ for all $z\in\mathbb{C}$, i.e. that $E_{a,z}$ is a generalized eigenfunction of the Laplacian. Moreover, $E_{a,z}$ is independent of the choice of $\varepsilon$. This follows by taking the difference of two versions of $E_{a,z}$ that correspond to two different values of $\varepsilon$ and then use meromorphicity together with the fact that the difference belongs to $L^{2}(\mathbb{M})$. 

Similarly to \cite{HRS}, we can proceed to the Fourier expansion of the generalized eigenfunction on the cusp $Y$. Since the Hankel functions form a fundamental set for the Bessel equation, and since $\varepsilon$ in the construction of $E_{a,z}$ can be arbitrary small, on $Y$ for any $z\in\mathbb{C}$ we have that
\begin{gather}\label{eexp}
E_{a,z}(x,y)=(a+x)^{\frac{a+1}{2}}H_{\frac{a+1}{2}}^{(2)}((a+x)e^{\frac{z}{2}})+C_{a}(z)(a+x)^{\frac{a+1}{2}}H_{\frac{a+1}{2}}^{(1)}((a+x)e^{\frac{z}{2}}) +\Psi_{a,z}(x,y),
\end{gather}
with a meromorphic in $z\in\mathbb{C}$ family $\Psi_{a,z}$ and a meromorphic in $z\in\mathbb{C}$ function $C_{a}(z)$ which is called {\em stationary scattering matrix}. Note that $C_{a}(z)$ and $\Psi_{a,z}$ are determined uniquely by the properties of $E_{a,z}$ and the choice of the fundamental set for the Bessel equation in the expansion \eqref{eexp}. 

When $\mathrm{Im}(z)\in(0,2\pi)$, since $\Psi_{a,z}\in L^{2}(\mathbb{Y})$, the Fourier coefficients of $\Psi_{a,z}$ are $L^{2}([1,\infty),(1+\frac{x}{a})^{-a}dx)$ solutions of (\ref{ee2}) for $n\neq0$. If we change variables by $x=t-a$ and $u(t-a)=t^{\frac{a}{2}}w(t)$ in (\ref{ee2}), then the equation obtains the form of \cite[(9a)]{HRS} (in the simple case of $n=2$ and $p=0$). The asymptotic behavior at infinity of the $L^{2}$-solution of the equation \cite[(9a)]{HRS} is explicitly given in \cite[Theorem 1.2]{HRS}. Hence, for the tail term $\Psi_{a,z}$ we have the following behavior on $Y$ as $x\rightarrow\infty$, namely 
\begin{gather*}
\Psi_{a,z}(x,y)=O(x^{\frac{a}{2}}e^{(-\frac{2\pi}{(a+1)a^{a}}+\varepsilon)(a+x)^{a+1}}), \quad \forall \varepsilon>0.
\end{gather*}
We can then summarize to the following.
\begin{theorem}\label{tb}
There exists a meromorphic in $z\in\mathbb{C}$ family $E_{a,z}$ of smooth functions on $M$ such that $(\Delta_{\mathbb{M}}-e^{z})E_{a,z}=0$ on $M$ for all $z\in\mathbb{C}$. Further, $E_{a,z}$ admits an asymptotic expansion on the cusp $Y$ given by
\[
E_{a,z}(x,y)=(a+x)^{\frac{a+1}{2}}H_{\frac{a+1}{2}}^{(2)}((a+x)e^{\frac{z}{2}})+C_{a}(z)(a+x)^{\frac{a+1}{2}}H_{\frac{a+1}{2}}^{(1)}((a+x)e^{\frac{z}{2}})+\Psi_{a,z}(x,y),
\]
for some meromorphic in $z\in\mathbb{C}$ function $C_{a}(z)$ called the {\em scattering matrix} and a tail term $\Psi_{a,z}(x,y)$ that satisfies 
\[
\Psi_{a,z}(x,y)=O(x^{\frac{a}{2}}e^{(-\frac{2\pi}{(a+1)a^{a}}+\varepsilon)(a+x)^{a+1}}), \quad \forall \varepsilon>0, \quad \mbox{when} \quad \mathrm{Im}(z)\in(0,2\pi).
\]
Moreover, $E_{a,z}$, $C_{a}(z)$ and $\Psi_{a,z}$ are uniquely determined by the above properties. The poles of $E_{a,z}$ are simple and contained in the following three sets:
\begin{itemize}
\item[(i)] A discrete set of points that is contained in $\cup_{k\in\mathbb{Z}}\{z\in\mathbb{C}\, |\, \mathrm{Im}(z)=2k\pi\}$.
\item[(ii)] The set $\mathcal{H}_{a}$ defined in \eqref{B}. 
\item[(iii)] The poles over $z\in\mathbb{C}$ of the family $(I+K_{a}(z))^{-1}$ defined in Theorem \ref{continuation}.
\end{itemize}
\end{theorem}
We can use the uniqueness from the above theorem in order to prove the functional equation of the scattering matrix as in \cite[Theorem 1.3]{HRS}. More precisely, we have that
\begin{eqnarray*}
\lefteqn{E_{a,z-2\pi i}(x,y)}\\
&=&(a+x)^{\frac{a+1}{2}}H_{\frac{a+1}{2}}^{(2)}((a+x)e^{\frac{z}{2}-\pi i})+C_{a}(z-2\pi i)(a+x)^{\frac{a+1}{2}}H_{\frac{a+1}{2}}^{(1)}((a+x)e^{\frac{z}{2}-\pi i}) +\Psi_{z-2\pi i}(x,y),
\end{eqnarray*}
where by using the properties of the Hankel functions 
\begin{gather*}
H_{\nu}^{(1)}(e^{-i\pi}\tau)=2\cos(\pi\nu)H_{\nu}^{(1)}(\tau)+e^{-i\pi\nu}H_{\nu}^{(2)}(\tau) \quad \mbox{and} \quad H_{\nu}^{(2)}(e^{-i\pi}\tau)=-e^{i\pi\nu}H_{\nu}^{(1)}(\tau), 
\end{gather*}
valid for any $\tau$ on the logarithmic cover of $\mathbb{C}$ and any $\nu\in\mathbb{C}$, we obtain
\begin{eqnarray}\nonumber
\lefteqn{E_{a,z-2\pi i}(x,y)=e^{-i\pi\frac{a+1}{2}}C_{a}(z-2\pi i)(a+x)^{\frac{a+1}{2}}H_{\frac{a+1}{2}}^{(2)}((a+x)e^{\frac{z}{2}})}\\\label{475}
&&+\big(2\cos(\pi\frac{a+1}{2})C_{a}(z-2\pi i)-e^{i\pi\frac{a+1}{2}}\big)(a+x)^{\frac{a+1}{2}}H_{\frac{a+1}{2}}^{(1)}((a+x)e^{\frac{z}{2}})+\Psi_{z-2\pi i}(x,y).
\end{eqnarray}
Moreover,
\begin{eqnarray*}
\lefteqn{e^{-i\pi\frac{a+1}{2}}C_{a}(z-2\pi i)E_{a,z}(x,y)=e^{-i\pi\frac{a+1}{2}}C_{a}(z-2\pi i)(a+x)^{\frac{a+1}{2}}H_{\frac{a+1}{2}}^{(2)}((a+x)e^{\frac{z}{2}})}\\
&&+e^{-i\pi\frac{a+1}{2}}C_{a}(z-2\pi i)C_{a}(z)(a+x)^{\frac{a+1}{2}}H_{\frac{a+1}{2}}^{(1)}((a+x)e^{\frac{z}{2}}) +e^{-i\pi\frac{a+1}{2}}C_{a}(z-2\pi i)\Psi_{a,z}(x,y).
\end{eqnarray*}
By comparing this equation with \eqref{475}, by uniqueness of Theorem \ref{tb}, we find the following. 
\begin{theorem}
The scattering matrix defined in Theorem \ref{tb} satisfies the following functional equation
\[
C_{a}(z)\big(C_{a}(z+2\pi i)+e^{i\pi a}-1\big)=e^{i\pi a}, \quad \forall z\in\mathbb{C}.
\]
\end{theorem} 
\begin{remark}
In a similar way to \cite[Section 4]{HRS} we can also prove unitarity for $C_{a}(z)$ as in \cite[Theorem 1.3]{HRS}. 
\end{remark}

\section{Approaching a surface with hyperbolic cusp}

In this section we show that our parametric family of surfaces approaches - from scattering point of view and in a certain sense - a surface with hyperbolic cusp. Under this consideration we are able to obtain scattering information for such a surface. By a surface with hyperbolic cusp we mean a two-dimensional Riemannian manifold consisting of a compact surface with boundary $S^{1}$ and a non-compact end equal to $[e,\infty) \times S^{1}$, such that when the Riemannian metric is restricted to the second part it admits the usual hyperbolic form, namely $(dt^{2}+dy^{2})/t^{2}$, $(t,y)\in [e,\infty)\times S^{1}$. Therefore, after changing of variables, we write such a surface as $(M=X\cup_{S^{1}}Y,g_{\mathbb{H}})$, such that when the Riemannian metric $g_{\mathbb{H}}$ is restricted to $Y$ it takes the form $dx^{2}+e^{-2x}dy$, $(x,y)\in [1,\infty)\times S^{1}$. Let $\mathbb{M}_{\mathbb{H}}=(M,g_{\mathbb{H}})$, $\mathbb{X}_{\mathbb{H}}=(X,g_{\mathbb{H}}|_{X})$ and $\mathbb{Y}_{\mathbb{H}}=(Y,g_{\mathbb{H}}|_{Y})$.
 
The hyperbolic Laplacian $\Delta_{\mathbb{H}}$ has a unique closed extension $\underline{\Delta}_{\mathbb{H}}$ in the space $L^{2}(\mathbb{M}_{\mathbb{H}})$ of square integrable functions with respect to the Riemannian measure induced by $g_{\mathbb{H}}$. There exists a unique generalized $s(1-s)$-eigenfunction $F_{\mathbb{H},s}$ of $\Delta_{\mathbb{H}}$ that has similar properties to $E_{a,z}$ as described in Theorem \ref{tb}, and on the cusp $Y$ it admits the following Fourier expansion, namely
\begin{gather}\label{heq}
F_{\mathbb{H},s}(x,y)=e^{xs}+S(s)e^{x(1-s)}+\Psi_{\mathbb{H},s}(x,y), 
\end{gather}
with
$$
 \Psi_{\mathbb{H},s}(x,y)=O(e^{-2\pi e^{x}}) \quad \mbox{when}\quad \mathrm{Re}(s)>\frac{1}{2}, \quad s\notin(\frac{1}{2},1],
$$
and the scattering matrix $S(s)$, $s\in\mathbb{C}$ (see e.g. \cite[(13)]{Mu}). For further details of the related spectral and scattering theory see also \cite{Iw}, \cite{LP}, \cite{Mu}, \cite{Mu2} and \cite{Mu3}. The following result is an immediate consequence of the choice of the metric $g$.

\begin{theorem}
For any $\phi\in C_{0}^{\infty}(Y\backslash\partial Y)$ and $\lambda\in \rho(\underline{\Delta}_{\mathbb{H}})$ we have that
\[
\|(\underline{\Delta}_{\mathbb{H}}-\lambda)^{-1}(\Delta_{\mathbb{M}}-\lambda)\phi-\phi\|_{L^{2}(\mathbb{M}_{\mathbb{H}})}\rightarrow0\quad \mbox{as} \quad a\rightarrow\infty.
\]
\end{theorem}

\begin{proof}

We start with
\begin{gather}\label{comparing2}
(\underline{\Delta}_{\mathbb{H}}-\lambda)^{-1}(\Delta_{\mathbb{M}}-\lambda)\phi=\phi+(\underline{\Delta}_{\mathbb{H}}-\lambda)^{-1}(\Delta_{\mathbb{M}}-\Delta_{\mathbb{H}})\phi.
\end{gather}
Since by \eqref{Delta} the coefficients of $\Delta_{\mathbb{M}}$ converge to the coefficients of $\Delta_{\mathbb{H}}$ uniformly on compact sets of the cuspidal part $Y$ when $a\rightarrow\infty$, the $L^{2}(\mathbb{M}_{\mathbb{H}})$-norm of $(\underline{\Delta}_{\mathbb{H}}-\lambda)^{-1}(\Delta_{\mathbb{M}}-\Delta_{\mathbb{H}})\phi$ tends to zero as $a\rightarrow\infty$. Hence, the result follows by \eqref{comparing2}.
\end{proof}

For any $z\in\mathbb{C}$ we write $z=z_{0}+i2k_{z}\pi$, with $\mathrm{Im}(z_{0})\in[-\pi,\pi)$ and $k_{z}\in\mathbb{Z}$. Further, let $\rho:(0,\infty)\rightarrow [1,\infty)$ be any function such that $\rho(a)\rightarrow \infty$ as $a\rightarrow \infty$, and define
\begin{gather*}
l(a)=2^{a}\Gamma(\frac{a+1}{2})\rho(a), \quad a\in(0,\infty),
\end{gather*}
where by $\Gamma$ we denote the gamma function. According to Theorem \ref{tb}, for any $a\notin\{2n+1\, |\, n\in\mathbb{N}\}$ we define the following generalized eigenfunction of $\Delta_{\mathbb{M}}$, namely 
\begin{gather*}
F_{a,z}=\eta_{a}(z)E_{a,z+2\ln(l(a))}
\end{gather*}
where
\begin{eqnarray*}
\lefteqn{\eta_{a}(z)=}\\
&&a^{-\frac{a}{2}}\sqrt{\frac{\pi l(a)e^{\frac{z_{0}}{2}}}{2}}e^{-ial(a)e^{\frac{z_{0}}{2}}}\Big(\frac{\sin(k_{z}\pi\frac{a+1}{2})}{\cos(\frac{\pi a}{2})}e^{i\frac{\pi a}{4}} +iC_{a}(z+2\ln(l(a)))\frac{\sin((k_{z}-1)\pi\frac{a+1}{2})}{\cos(\frac{\pi a}{2})}e^{-i\frac{\pi a}{4}}\Big)^{-1}.
\end{eqnarray*}
The role of $\eta_{a}$ and of the spectral shift $z\mapsto z+2\ln(l(a))$ is to achieve uniform convergence on compact sets of $Y$ of the zero $S^{1}$-Fourier coefficient (except possibly of the scattering matrix) of $F_{a,z}$ to the zero $S^{1}$-Fourier coefficient of $F_{\mathbb{H},s}$ for large values of the spectral parameter as $a\rightarrow\infty$. By recalling that 
$(1+\frac{x}{a})^{a}\rightarrow e^{x}$ uniformly in $x$ lying on a compact subset of $[1,\infty)$ as $a\rightarrow\infty$, and that the non-zero $S^{1}$-Fourier coefficients of $F_{a,z}$ are given according to Theorem \ref{tb} by the tail term $\eta_{a}(z)\Psi_{a,z+2\ln(l(a))}$, we have the following result.

\begin{theorem}\label{tap}
Under the spectral equivalence $s=\frac{1}{2}+il(a)e^{\frac{z_{0}}{2}}$, for any 
\begin{gather*}\label{z}
z\in\mathbb{C}\backslash\bigcup_{k\in\mathbb{Z}}\Big\{\mu\in\mathbb{C}\, |\, \mathrm{Im}(\mu+i2k\pi)=\pi\Big\}
\end{gather*}
we have that
\begin{gather*}
F_{a,z}-\eta_{a}(z)\Psi_{a,z+2\ln(l(a))}=e^{xs}P_{a}(z_{0},x)+\xi_{a}(z)e^{x(1-s)}Q_{a}(z_{0},x)
\end{gather*}
on $Y$, where
$$
P_{a}(z_{0},x)=(1+\frac{x}{a})^{\frac{a}{2}}e^{-\frac{x}{2}}p_{a}(z_{0},x) \quad \text{and} \quad Q_{a}(z_{0},x)=(1+\frac{x}{a})^{\frac{a}{2}}e^{-\frac{x}{2}}q_{a}(z_{0},x)
$$
for some smooth in $x$ and holomorphic in $w_{0}$, $(x,w_{0})\in[1,\infty)\times\{\mu\in\mathbb{C}\, |\, |\mathrm{Im}(\mu)|<\pi\}$, functions $p_{a}(w_{0},x)$, $q_{a}(w_{0},x)$, and 
\begin{gather*}
\xi_{a}(z)=\frac{i\sin((k_{z}+1)\pi\frac{a+1}{2})e^{i\frac{\pi a}{2}}-C_{a}(z+2\ln(l(a)))\sin(k_{z}\pi\frac{a+1}{2})}{\sin(k_{z}\pi\frac{a+1}{2})e^{i\frac{\pi a}{2}} +iC_{a}(z+2\ln(l(a)))\sin((k_{z}-1)\pi\frac{a+1}{2})}e^{-i2al(a)e^{\frac{z_{0}}{2}}},
\end{gather*}
$a\notin\{2n+1\, |\, n\in\mathbb{N}\}$. Furthermore, for any $c\in\mathbb{R}$ and $\varepsilon\in(0,\pi)$ we have that $p_{a}(w_{0},x)\rightarrow 1$ and $q_{a}(w_{0},x)\rightarrow 1$ as $a\rightarrow \infty$ uniformly in $x$ and $w_{0}$ with $x\in[1,\infty)$ and 
\[
w_{0}\in \Big\{\mu\in\mathbb{C}\, |\, \mathrm{Re}(\mu)\geq c\Big\}\cap\Big\{\mu\in\mathbb{C}\, |\, |\mathrm{Im}(\mu)|\leq\pi-\varepsilon\Big\}.
\]
\end{theorem}
\begin{proof}
By \cite[3.62 (5)-(6)]{Wa} we have that 
\begin{gather*}\label{H1}
H_{\nu}^{(1)}(e^{w})=-\frac{\sin((m-1)\pi \nu)}{\sin(\pi\nu)}H_{\nu}^{(1)}(e^{w_{+}})-e^{-\nu\pi i}\frac{\sin(m\pi \nu)}{\sin(\pi\nu)}H_{\nu}^{(2)}(e^{w_{+}})
\end{gather*}
and
\begin{gather*}\label{H2}
H_{\nu}^{(2)}(e^{w})=e^{\nu\pi i}\frac{\sin(m\pi \nu)}{\sin(\pi\nu)}H_{\nu}^{(1)}(e^{w_{+}})+\frac{\sin((m+1)\pi \nu)}{\sin(\pi\nu)}H_{\nu}^{(2)}(e^{w_{+}}),
\end{gather*}
where $w=w_{+}+im\pi$, with $\mathrm{Im}(w_{+})\in[0,\pi)$, $m\in\mathbb{Z}$ and $\nu>0$. Similarly
\begin{gather*}\label{H1}
H_{\nu}^{(1)}(e^{w})=-\frac{\sin((n-1)\pi \nu)}{\sin(\pi\nu)}H_{\nu}^{(1)}(e^{w_{-}})-e^{-\nu\pi i}\frac{\sin(n\pi \nu)}{\sin(\pi\nu)}H_{\nu}^{(2)}(e^{w_{-}})
\end{gather*}
and
\begin{gather*}\label{H2}
H_{\nu}^{(2)}(e^{w})=e^{\nu\pi i}\frac{\sin(n\pi \nu)}{\sin(\pi\nu)}H_{\nu}^{(1)}(e^{w_{-}})+\frac{\sin((n+1)\pi \nu)}{\sin(\pi\nu)}H_{\nu}^{(2)}(e^{w_{-}}),
\end{gather*}
where $w=w_{-}+in\pi$, with $\mathrm{Im}(w_{-})\in[-\pi,0)$ and $n\in\mathbb{Z}$. Hence, by restricting $\nu>\frac{1}{2}$, from the above formulas, \cite[6.12 (3)-(4)]{Wa} 
and the formula
$$
(1+\frac{it}{2\lambda})^{\nu-\frac{1}{2}}=1+(\frac{1}{2}-\nu)\frac{t}{2i\lambda}\int_{0}^{1}(1-\frac{ty}{2i\lambda})^{\nu-\frac{3}{2}}dy, \quad t\geq0, \quad \lambda\in\mathbb{C}\backslash\{0\}
$$
(see e.g. \cite[7.2]{Wa}), for any $w\in\mathbb{C}$ with $\mathrm{Im}(w_{0})\in(-\pi,\pi)$ we obtain the following integral representations for the Hankel functions
\begin{eqnarray}\nonumber
\lefteqn{H_{\nu}^{(1)}(e^{\frac{w}{2}})=}\\\nonumber
&&-\frac{\sin((k_{w}-1)\pi\nu)}{\sin(\pi\nu)}\sqrt{\frac{2}{\pi e^{\frac{w_{0}}{2}}}}e^{i(e^{\frac{w_{0}}{2}}-\frac{\pi\nu}{2}-\frac{\pi}{4})}(1+Q_{\nu}^{+}(e^{\frac{w_{0}}{2}}))\\\label{H1as}
&&-\frac{\sin(k_{w}\pi\nu)}{\sin(\pi\nu)}\sqrt{\frac{2}{\pi e^{\frac{w_{0}}{2}}}}e^{i(\frac{\pi}{4}-\frac{\pi\nu}{2}-e^{\frac{w_{0}}{2}})}(1+Q_{\nu}^{-}(e^{\frac{w_{0}}{2}}))
\end{eqnarray}
and
\begin{eqnarray}\nonumber
\lefteqn{H_{\nu}^{(2)}(e^{\frac{w}{2}})=}\\\nonumber
&&\frac{\sin(k_{w}\pi\nu)}{\sin(\pi\nu)}\sqrt{\frac{2}{\pi e^{\frac{w_{0}}{2}}}}e^{i(e^{\frac{w_{0}}{2}}+\frac{\pi\nu}{2}-\frac{\pi}{4})}(1+Q_{\nu}^{+}(e^{\frac{w_{0}}{2}}))\\\label{H2as}
&&+\frac{\sin((k_{w}+1)\pi\nu)}{\sin(\pi\nu)}\sqrt{\frac{2}{\pi e^{\frac{w_{0}}{2}}}}e^{i(\frac{\pi\nu}{2}+\frac{\pi}{4}-e^{\frac{w_{0}}{2}})}(1+Q_{\nu}^{-}(e^{\frac{w_{0}}{2}})),
\end{eqnarray}
where
$$
Q_{\nu}^{\pm}(\lambda)=\pm\frac{\frac{1}{2}-\nu}{2 i\lambda\Gamma(\nu+\frac{1}{2})}\int_{0}^{+\infty}e^{-t}t^{\nu+\frac{1}{2}}\big(\int_{0}^{1}(1\mp\frac{ty}{2i\lambda})^{\nu-\frac{3}{2}}dy\big)dt, \quad \lambda\in\mathbb{C},\quad \mathrm{arg}(\lambda)\in(-\frac{\pi}{2},\frac{\pi}{2}).
$$
Therefore, by \eqref{H1as}-\eqref{H2as} and Theorem \ref{tb} we obtain that
\begin{eqnarray*}
\lefteqn{E_{a,z+2\ln(l(a))}-\Psi_{a,z+2\ln(l(a))}}\\
&=& (a+x)^{\frac{a+1}{2}}H_{\frac{a+1}{2}}^{(2)}((a+x)l(a)e^{\frac{z}{2}})+C_{a}(z+2\ln(l(a)))(a+x)^{\frac{a+1}{2}}H_{\frac{a+1}{2}}^{(1)}((a+x)l(a)e^{\frac{z}{2}}) \\
&=&(a+x)^{\frac{a}{2}}\sqrt{\frac{2}{\pi l(a)e^{\frac{z_{0}}{2}}}}\\
&&\times\Big(\frac{\sin(k_{z}\pi\frac{a+1}{2})}{\cos(\frac{\pi a}{2})}e^{i((a+x)l(a)e^{\frac{z_{0}}{2}}+\frac{\pi a}{4})}(1+Q_{\frac{a+1}{2}}^{+}((a+x)l(a)e^{\frac{z_{0}}{2}}))\\
&&+i\frac{\sin((k_{z}+1)\pi\frac{a+1}{2})}{\cos(\frac{\pi a}{2})}e^{-i((a+x)l(a)e^{\frac{z_{0}}{2}}-\frac{\pi a}{4})}(1+Q_{\frac{a+1}{2}}^{-}((a+x)l(a)e^{\frac{z_{0}}{2}}))\Big)\\
&&+C_{a}(z+2\ln(l(a)))(a+x)^{\frac{a}{2}}\sqrt{\frac{2}{\pi l(a)e^{\frac{z_{0}}{2}}}}\\
&&\times\Big(i\frac{\sin((k_{z}-1)\pi\frac{a+1}{2})}{\cos(\frac{\pi a}{2})}e^{i((a+x)l(a)e^{\frac{z_{0}}{2}}-\frac{\pi a}{4})}(1+Q_{\frac{a+1}{2}}^{+}((a+x)l(a)e^{\frac{z_{0}}{2}}))\\
&&-\frac{\sin(k_{z}\pi\frac{a+1}{2})}{\cos(\frac{\pi a}{2})}e^{-i((a+x)l(a)e^{\frac{z_{0}}{2}}+\frac{\pi a}{4})}(1+Q_{\frac{a+1}{2}}^{-}((a+x)l(a)e^{\frac{z_{0}}{2}}))\Big),
\end{eqnarray*}
when $\mathrm{Im}(z_{0})\in(-\pi,\pi)$. By rearranging the above equation we get that 
\begin{eqnarray}\nonumber
\lefteqn{E_{a,z+2\ln(l(a))}-\Psi_{a,z+2\ln(l(a))}}\\\nonumber
&=&(a+x)^{\frac{a}{2}}\sqrt{\frac{2}{\pi l(a)e^{\frac{z_{0}}{2}}}}e^{i(a+x)l(a)e^{\frac{z_{0}}{2}}}(1+Q_{\frac{a+1}{2}}^{+}((a+x)l(a)e^{\frac{z_{0}}{2}}))\\\nonumber
&&\times\Big(\frac{\sin(k_{z}\pi\frac{a+1}{2})}{\cos(\frac{\pi a}{2})}e^{i\frac{\pi a}{4}} +iC_{a}(z+2\ln(l(a)))\frac{\sin((k_{z}-1)\pi\frac{a+1}{2})}{\cos(\frac{\pi a}{2})}e^{-i\frac{\pi a}{4}}\Big)\\\nonumber
&&+(a+x)^{\frac{a}{2}}\sqrt{\frac{2}{\pi l(a)e^{\frac{z_{0}}{2}}}}e^{-i(a+x)l(a)e^{\frac{z_{0}}{2}}}(1+Q_{\frac{a+1}{2}}^{-}((a+x)l(a)e^{\frac{z_{0}}{2}}))\\\label{ef}
&&\times\Big(i\frac{\sin((k_{z}+1)\pi\frac{a+1}{2})}{\cos(\frac{\pi a}{2})}e^{i\frac{\pi a}{4}}-C_{a}(z+2\ln(l(a)))\frac{\sin(k_{z}\pi\frac{a+1}{2})}{\cos(\frac{\pi a}{2})}e^{-i\frac{\pi a}{4}}\Big).
\end{eqnarray}
Concerning the $Q_{\frac{a+1}{2}}^{\pm}((a+x)l(a)e^{\frac{z_{0}}{2}})$ terms, when $a>2$ we estimate
\begin{eqnarray}\nonumber
\lefteqn{|Q_{\frac{a+1}{2}}^{\pm}((a+x)l(a)e^{\frac{z_{0}}{2}})|}\\\nonumber
&\leq&\frac{a}{4(a+x)l(a)\Gamma(\frac{a}{2}+1)|e^{\frac{z_{0}}{2}}|}\int_{0}^{+\infty}e^{-t}t^{\frac{a}{2}+1}(1+\frac{t}{2(a+x)l(a)|e^{\frac{z_{0}}{2}}|})^{\frac{a}{2}-1}dt\\\nonumber
&\leq&\frac{a}{4(a+x)l(a)\Gamma(\frac{a}{2}+1)|e^{\frac{z_{0}}{2}}|}(1+\frac{1}{2(a+x)l(a)|e^{\frac{z_{0}}{2}}|})^{\frac{a}{2}-1}(\int_{0}^{1}e^{-t}t^{\frac{a}{2}+1}dt+\int_{1}^{+\infty}e^{-t}t^{a}dt)\\\nonumber
&\leq&\frac{a}{4(a+x)l(a)\Gamma(\frac{a}{2}+1)|e^{\frac{z_{0}}{2}}|}(1+\frac{1}{2(a+x)l(a)|e^{\frac{z_{0}}{2}}|})^{\frac{a}{2}-1}(\int_{0}^{
+\infty}e^{-t}t^{\frac{a}{2}+1}dt+\int_{0}^{+\infty}e^{-t}t^{a}dt)\\\nonumber
&=&\frac{a(\Gamma(\frac{a}{2}+2)+\Gamma(a+1))}{4(a+x)l(a)\Gamma(\frac{a}{2}+1)|e^{\frac{z_{0}}{2}}|}(1+\frac{1}{2(a+x)l(a)|e^{\frac{z_{0}}{2}}|})^{\frac{a}{2}-1}\\\label{Qterm}
&=&\frac{a(1+\frac{a}{2}+\frac{2^{a}}{\sqrt{\pi}}\Gamma(\frac{a+1}{2}))}{4(a+x)l(a)|e^{\frac{z_{0}}{2}}|}(1+\frac{1}{2(a+x)l(a)|e^{\frac{z_{0}}{2}}|})^{\frac{a}{2}-1},
\end{eqnarray}
where we have used the functional equation of the gamma function and the duplication formula 
$$
\Gamma(a)\Gamma(a+\frac{1}{2})=2^{1-2a}\sqrt{\pi}\Gamma(2a).
$$
Therefore, for any $c\in\mathbb{R}$ and $\varepsilon\in(0,\pi)$ the right hand side of \eqref{Qterm} tends to zero as $a$ tends to infinity uniformly in $x$ and $z_{0}$ with $x\in[1,\infty)$ and 
\begin{gather*}
z_{0}\in \Big\{\mu\in\mathbb{C}\, |\, \mathrm{Re}(\mu)\geq c\Big\}\cap\Big\{\mu\in\mathbb{C}\, |\, |\mathrm{Im}(\mu)|\leq\pi-\varepsilon\Big\}.
\end{gather*}
The result now follows by \eqref{ef} and the choice of the factor $\eta_{a}(z)$.
\end{proof}

\section{The zeros of the Hankel functions for large order}

When the parameter $a$ tends to infinity the metric $g|_{Y}$ converges to the hyperbolic metric having the results of the previous section as a consequence. The spectral equivalence from Theorem \ref{tap} becomes $s=\frac{1}{2}+i(-1)^{k_{\lambda}}e^{\frac{\lambda}{2}}$ under $\lambda=z+2\ln(l(a))$. Therefore, it is of particular interest to study the trajectories of the poles with respect to $\lambda$ of the family 
\begin{gather}
\omega_{a}(\lambda)=\frac{i\sin((k_{\lambda}+1)\pi\frac{a+1}{2})e^{i\frac{\pi a}{2}}-C_{a}(\lambda)\sin(k_{\lambda}\pi\frac{a+1}{2})}{\sin(k_{\lambda}\pi\frac{a+1}{2})e^{i\frac{\pi a}{2}} +iC_{a}(\lambda)\sin((k_{\lambda}-1)\pi\frac{a+1}{2})}e^{(-1)^{1+k_{\lambda}}i2ae^{\frac{\lambda}{2}}}
\end{gather}
as $a\rightarrow \infty$, $a\notin\{2n+1\, |\, n\in\mathbb{N}\}$. If we assume that $a=4n$, $n\in\mathbb{N}$, then by simplifying the above formula, it is sufficient to study the trajectories of the poles with respect to $\lambda$ of the family
\begin{gather}\label{scattapproach}
\phi_{n}(\lambda)=\frac{\sin((k_{\lambda}+1)\frac{\pi}{2})+iC_{4n}(\lambda)\sin(k_{\lambda}\frac{\pi}{2})}{\sin(k_{\lambda}\frac{\pi}{2})+iC_{4n}(\lambda)\sin((k_{\lambda}-1)\frac{\pi}{2})}
\end{gather}
as $n\rightarrow \infty$, $n\in\mathbb{N}$.

Since the poles of $C_{a}(\lambda)$ that lie on $\cup_{k\in\mathbb{Z}}\{\mu\in\mathbb{C}\, |\, \mathrm{Im}(\mu)=2k\pi\}$ remain in this set, according to Theorem \ref{tb}, the interesting case is to study the flow with respect to $a$ of the set $\mathcal{H}_{a}$ defined in \eqref{B} together with the poles of the family $(I+K_{a}(\lambda))^{-1}$. We can use the asymptotic expansion of the Hankel functions of large order to obtain information about the flow of the points in $\mathcal{H}_{a}$. This expansion for the Hankel function of the first kind of integer order is given by \cite[(9.3)]{Ol}, and in \cite[Section 9]{Ol} the following is shown.
\begin{proposition}\label{p1}
The zeros in $w$ of $H_{2n+\frac{1}{2}}^{(1)}((2n+\frac{1}{2}) w)$ in the logarithmic cover of $\mathbb{C}$ when $n\rightarrow\infty$, $n\in \mathbb{N}$, stay arbitrary closed to the bounded curve $\partial{\bf K}$ described in \cite[(4.10)]{Ol} and its conjugate $\overline{\partial{\bf K}}$. 
\end{proposition}

\begin{remark}
From \eqref{scattapproach} the poles of $C_{4n}(\lambda)$ that lie in the set $\{\mu\in\mathbb{C}\, |\, |\mathrm{Im}(\mu)|<\pi\}$, i.e. that have $k_{\lambda}=0$, are expected to accumulate to zeros of $S(s)$. Hence, by Proposition \ref{p1} and the fact that the zeros of the Hankel function of the second kind are the complex conjugates of the zeros of the Hankel function of the first kind, the poles of $C_{4n}(\lambda)$ that are contained in the sets described by (i) and (ii) of Theorem \ref{tb}
are expected to accumulate to zeros of $S(s)$ in the axis $\frac{1}{2}+i\mathbb{R}$ and in the set $\frac{1}{2}+\frac{i}{2}(\partial{\bf K}\cup\overline{\partial{\bf K}})$. However, the behavior of the poles of the family $(I+K_{a}(\lambda))^{-1}$ when $a$ tends to infinity is not clear. 
\end{remark}

\subsection*{Acknowledgment} We thank the referee for giving helpful suggestions that improved the document.

\end{document}